%% file: manysurfaces.tex
\documentclass[a4paper]{article}
\usepackage{amsmath,amsfonts,amssymb,amsthm,pstricks,caption}

\newtheorem{lemma}{Lemma}
\newtheorem{theorem}[lemma]{Theorem}
\newtheorem{proposition}[lemma]{Proposition}
\newtheorem{corollary}[lemma]{Corollary}

\theoremstyle{definition}
\newtheorem{definition}[lemma]{Definition}

\begin{document}

\title{Knots with many minimal genus Seifert surfaces}
\author{Jessica E. Banks}
\date{}
\maketitle


In \cite{2013arXiv1308.2899R}, Roberts proves the following result using sutured Floer homology and the Seifert form.

\begin{theorem}[\cite{2013arXiv1308.2899R} Theorem 1]
For $n\in\mathbb{N}$, there is an oriented prime knot $K_n$ in $\mathbb{S}^3$ of genus $n$ that has at least $2^{2n-1}$ distinct minimal genus Seifert surfaces, up to ambient isotopy in $\mathbb{S}^3$.
\end{theorem}

We will only consider oriented links.
Roberts notes that the knots $K_n$ are all arborescent and alternating.
Minimal genus Seifert surfaces have previously been classified up to ambient isotopy in the link complement\footnote{Thank you to Lawrence Roberts for pointing out this distinction in definition.} for special arborescent links by Sakuma (\cite{MR1315011}), and for special alternating links by the author (\cite{2011arXiv1106.3180B}, see also \cite{MR1664976}). 
A link is \textit{special arborescent} if it is the boundary of a surface given by plumbing annuli in a particular fashion defined by a labelled tree, where each annulus is embedded in $\mathbb{S}^3$ such that the core curve is an unknot and the annulus is incompressible.
On the other hand, a link is \textit{special alternating} if it has an alternating diagram $D$ such that every Seifert circle of $D$ bounds a disc in $\mathbb{S}^2$ disjoint from the other Seifert circles.

Here we give a family of prime, special arborescent, special alternating knots $L_n$ that are known by geometric methods (for example by \cite{2011arXiv1106.3180B}) to have exactly $2^{2n-1}$ distinct minimal genus Seifert surfaces up to ambient isotopy in the knot complement. 
This count also follows from the classification of incompressible surfaces in rational knot complements by Hatcher and Thurston (\cite{MR778125}) as the knots given are also rational knots.
 
The knots $L_n$ are a subset of those considered by Roberts. It therefore follows that each knot $L_n$ has exactly $2^{2n-1}$ minimal genus Seifert surfaces up to isotopy in $\mathbb{S}^3$.

\begin{definition}
For $n\in\mathbb{N}$, build as follows a link $L_n$, a Seifert surface $R_n$ for $L_n$, and a planar graph $\mathcal{G}_n$.
Begin with Figure \ref{pic1}(a). Fill the boundary circle on the right with Figure \ref{pic1}(b).
\begin{figure}[htb]
\centering
(a)
\input{pica}
(b)
\input{picb}
\caption{\label{pic1}}
\end{figure}
If $n\neq 1$, repeat the following three steps $(n-1)$ times.
First fill each of the two boundary circles with Figure \ref{pic2}(a). Second, on the $m$th step, fill the circle on the left with Figure \ref{pic1}(b) $(2m)$ times.
Thirdly, on the $m$th step, fill the circle on the right with Figure \ref{pic1}(b) $(2m+1)$ times.
Finally, fill the remaining two circles with Figure \ref{pic2}(b).
\begin{figure}[htb]
\centering
(a)
\input{picc}
(b)
\input{picd}
\caption{\label{pic2}}
\end{figure}
\end{definition}

For $n\in\mathbb{N}$, the link $L_n$ is special alternating, and $R_n$ is minimal genus (by \cite{MR860665} Theorem 4) with Euler characteristic $1-2n$.
By \cite{MR721450} Theorem 1, each $L_n$ is prime. 
Note that the graph $\mathcal{G}_n$ has $2n-1$ connected components, each of which consists of two vertices joined by two edges.
For the family of links $L_n$, Theorem 1.3 of \cite{2011arXiv1106.3180B} gives the following.

\begin{proposition}
Let $A_n$ be the set of ways of choosing one of the two edges of each connected component of $\mathcal{G}_n$.
Then there is a bijection between $A_n$ and the set of ambient isotopy classes of minimal genus Seifert surfaces for $L_n$.
\end{proposition}

It remains only to establish the following.

\begin{lemma}
Each link $L_n$ is a knot.
\end{lemma}
\begin{proof}
If $L_n$ is a knot for some $n\in\mathbb{N}$, then so is $L_{n+2}$. In addition, $L_1$ and $L_2$ are both knots.
\end{proof}

\begin{corollary}
For $n\in\mathbb{N}$, there is an oriented prime knot $L_n$ in $\mathbb{S}^3$ of genus $n$ that has $2^{2n-1}$ distinct minimal genus Seifert surfaces, up to ambient isotopy in $\mathbb{S}^3$.
\end{corollary}

\bibliography{manysurfacesreferences}
\bibliographystyle{hplain}

\bigskip
\noindent
CIRGET, D\'{e}partement de math\'{e}matiques,
UQAM

\noindent
Case postale 8888, Centre-ville

\noindent
Montr\'eal, H3C 3P8,
Quebec, Canada

\noindent
\textit{jessica.banks[at]lmh.oxon.org}

\end{document}

%% file: pica.tex
\psset{xunit=.35pt,yunit=.35pt,runit=.35pt}
\begin{pspicture}(360,360)
{
\newgray{lightgrey}{0.9}
}
{
\pscustom[linestyle=none,fillstyle=solid,fillcolor=lightgray]
{
\newpath
\moveto(60,180)
\lineto(90,150)
\lineto(90,90)
\lineto(150,90)
\lineto(180,120)
\lineto(210,90)
\lineto(270,90)
\lineto(300,120)
\lineto(330,90)
\lineto(330,30)
\lineto(30,30)
\lineto(30,150)
\closepath
\moveto(180,240)
\lineto(150,210)
\lineto(180,180)
\lineto(210,210)
\closepath
\moveto(300,240)
\lineto(270,210)
\lineto(300,180)
\lineto(330,210)
\closepath
}
}
{
\pscustom[linestyle=none,fillstyle=solid,fillcolor=lightgrey]
{
\newpath
\moveto(60,180)
\lineto(90,210)
\lineto(90,270)
\lineto(150,270)
\lineto(180,240)
\lineto(210,270)
\lineto(270,270)
\lineto(300,240)
\lineto(330,270)
\lineto(330,330)
\lineto(30,330)
\lineto(30,210)
\closepath
\moveto(300,180)
\lineto(270,150)
\lineto(300,120)
\lineto(330,150)
\closepath
\moveto(180,180)
\lineto(150,150)
\lineto(180,120)
\lineto(210,150)
\closepath
}
}
{
\pscustom[linewidth=1.4,linecolor=black]
{
\newpath
\moveto(70,190)
\lineto(90,210)
\lineto(90,270)
\lineto(150,270)
\lineto(210,210)
\lineto(190,190)
\moveto(170,169.99998)
\lineto(150,149.99998)
\lineto(210,89.99998)
\lineto(270,89.99998)
\lineto(330,149.99998)
\lineto(270,209.99999)
\lineto(290,229.99999)
\moveto(310,249.99999)
\lineto(330,269.99999)
\lineto(330,329.99999)
\lineto(30,329.99999)
\lineto(30,269.99999)
\lineto(30,209.99999)
\lineto(90,149.99998)
\lineto(90,89.99998)
\lineto(150,89.99998)
\lineto(210,149.99998)
\lineto(150,209.99999)
\lineto(170,229.99999)
\moveto(190,249.99999)
\lineto(210,269.99999)
\lineto(270,269.99999)
\lineto(330,209.99999)
\lineto(310,189.99999)
\moveto(290,169.99998)
\lineto(270,149.99998)
\lineto(330,89.99998)
\lineto(330,29.99998)
\lineto(30,29.99998)
\lineto(30,149.99998)
\lineto(50,169.99998)
}
}
{
\pscustom[linewidth=1.4,linecolor=black,linestyle=dashed,dash=5.6 5.6,fillstyle=solid,fillcolor=white]
{
\newpath
\moveto(340,119.99999262)
\curveto(340,142.09138261)(322.09138999,159.99999262)(300,159.99999262)
\curveto(277.90861001,159.99999262)(260,142.09138261)(260,119.99999262)
\curveto(260,97.90860262)(277.90861001,79.99999262)(300,79.99999262)
\curveto(322.09138999,79.99999262)(340,97.90860262)(340,119.99999262)
\closepath
\moveto(220,119.99999262)
\curveto(220,142.09138261)(202.09138999,159.99999262)(180,159.99999262)
\curveto(157.90861001,159.99999262)(140,142.09138261)(140,119.99999262)
\curveto(140,97.90860262)(157.90861001,79.99999262)(180,79.99999262)
\curveto(202.09138999,79.99999262)(220,97.90860262)(220,119.99999262)
\closepath
}
}
{
\pscustom[linestyle=none,fillstyle=solid,fillcolor=gray]
{
\newpath
\moveto(190,299.99999738)
\curveto(190,294.47714988)(185.5228475,289.99999738)(180,289.99999738)
\curveto(174.4771525,289.99999738)(170,294.47714988)(170,299.99999738)
\curveto(170,305.52284488)(174.4771525,309.99999738)(180,309.99999738)
\curveto(185.5228475,309.99999738)(190,305.52284488)(190,299.99999738)
\closepath
\moveto(190,59.99999738)
\curveto(190,54.47714988)(185.5228475,49.99999738)(180,49.99999738)
\curveto(174.4771525,49.99999738)(170,54.47714988)(170,59.99999738)
\curveto(170,65.52284488)(174.4771525,69.99999738)(180,69.99999738)
\curveto(185.5228475,69.99999738)(190,65.52284488)(190,59.99999738)
\closepath
}
}
{
\pscustom[linewidth=2.8,linecolor=gray]
{
\newpath
\moveto(180,300)
\lineto(120,280)
\lineto(120,80)
\lineto(180,60)
\lineto(240,80)
\lineto(240,280)
\closepath
}
}
\end{pspicture}

%% file: picb.tex
\psset{xunit=.35pt,yunit=.35pt,runit=.35pt}
\begin{pspicture}(360,360)
{
\newgray{lightgrey}{0.9}
}
{
\pscustom[linestyle=none,fillstyle=solid,fillcolor=lightgray]
{
\newpath
\moveto(180,250)
\lineto(150,220)
\lineto(180,190)
\lineto(210,220)
\closepath
\moveto(160,10.00001)
\lineto(60,60.00001)
\lineto(150,100)
\lineto(180,130)
\lineto(210,100)
\lineto(300,60.00001)
\lineto(195,10.00001)
\closepath
}
}
{
\pscustom[linestyle=none,fillstyle=solid,fillcolor=lightgrey]
{
\newpath
\moveto(150,350)
\lineto(60,300)
\lineto(150,280)
\lineto(180,250)
\lineto(210,280)
\lineto(300,300)
\lineto(205,350)
\closepath
\moveto(180,190)
\lineto(150,160)
\lineto(180,130)
\lineto(210,160)
\closepath
}
}
{
\pscustom[linewidth=1.4,linecolor=black]
{
\newpath
\moveto(190,140)
\lineto(210,160)
\lineto(150,220)
\lineto(170,240)
\moveto(190,260)
\lineto(210,280)
\lineto(300,300)
\moveto(60,60)
\lineto(150,100)
\lineto(170,120)
\moveto(300,60)
\lineto(210,100)
\lineto(150,160)
\lineto(170,180)
\moveto(190,200)
\lineto(210,220)
\lineto(150,280)
\lineto(60,300)
}
}
{
\pscustom[linewidth=11,linecolor=white]
{
\newpath
\moveto(346.49664,180.00001328)
\curveto(346.49664,88.04645811)(271.95355517,13.50337328)(180,13.50337328)
\curveto(88.04644483,13.50337328)(13.50336,88.04645811)(13.50336,180.00001328)
\curveto(13.50336,271.95356845)(88.04644483,346.49665328)(180,346.49665328)
\curveto(271.95355517,346.49665328)(346.49664,271.95356845)(346.49664,180.00001328)
\closepath
}
}
{
\pscustom[linewidth=1.4,linecolor=black,linestyle=dashed,dash=5.6 5.6]
{
\newpath
\moveto(340,180.00000047)
\curveto(340,268.36556044)(268.36555997,340.00000047)(180,340.00000047)
\curveto(91.63444003,340.00000047)(20,268.36556044)(20,180.00000047)
\curveto(20,91.6344405)(91.63444003,20.00000047)(180,20.00000047)
\curveto(268.36555997,20.00000047)(340,91.6344405)(340,180.00000047)
\closepath
}
}
{
\pscustom[linewidth=1.4,linecolor=black,linestyle=dashed,dash=5.6 5.6,fillstyle=solid,fillcolor=white]
{
\newpath
\moveto(220,130.00000262)
\curveto(220,152.09139261)(202.09138999,170.00000262)(180,170.00000262)
\curveto(157.90861001,170.00000262)(140,152.09139261)(140,130.00000262)
\curveto(140,107.90861262)(157.90861001,90.00000262)(180,90.00000262)
\curveto(202.09138999,90.00000262)(220,107.90861262)(220,130.00000262)
\closepath
}
}
\end{pspicture}

%% file: picc.tex
\psset{xunit=.35pt,yunit=.35pt,runit=.35pt}
\begin{pspicture}(360,360)
{
\newgray{lightgrey}{0.9}
}
{
\pscustom[linestyle=none,fillstyle=solid,fillcolor=lightgray]
{
\newpath
\moveto(250,250)
\lineto(220,220)
\lineto(250,190)
\lineto(280,220)
\closepath
\moveto(160,10.00001)
\lineto(60,60.00001)
\lineto(100,160.00001)
\lineto(130,190.00001)
\lineto(160,160.00001)
\lineto(160,100.00001)
\lineto(220,100.00001)
\lineto(250,130.00001)
\lineto(280,100.00001)
\lineto(300,60.00001)
\lineto(195,10.00001)
\closepath
}
}
{
\pscustom[linestyle=none,fillstyle=solid,fillcolor=lightgrey]
{
\newpath
\moveto(150,350)
\lineto(60,300)
\lineto(100,220)
\lineto(130,190)
\lineto(160,220)
\lineto(160,280)
\lineto(220,280)
\lineto(250,250)
\lineto(300,300)
\lineto(205,350)
\closepath
\moveto(250,190.00001)
\lineto(220,160.00001)
\lineto(250,130.00001)
\lineto(280,160.00001)
\closepath
}
}
{
\pscustom[linewidth=1.4,linecolor=black]
{
\newpath
\moveto(60,300)
\lineto(100,220)
\lineto(160,160)
\lineto(160,100)
\lineto(220,100)
\lineto(240,120)
\moveto(260,140)
\lineto(280,160)
\lineto(220,220)
\lineto(240,240)
\moveto(260,260)
\lineto(300,300)
\moveto(120,180)
\lineto(100,160)
\lineto(60,60)
\moveto(300,60)
\lineto(280,100)
\lineto(220,160)
\lineto(240,180)
\moveto(260,200)
\lineto(280,220)
\lineto(220,280)
\lineto(160,280)
\lineto(160,220)
\lineto(140,200)
}
}
{
\pscustom[linewidth=1.4,linecolor=black,linestyle=dashed,dash=5.6 5.6,fillstyle=solid,fillcolor=white]
{
\newpath
\moveto(290,129.99999262)
\curveto(290,152.09138261)(272.09138999,169.99999262)(250,169.99999262)
\curveto(227.90861001,169.99999262)(210,152.09138261)(210,129.99999262)
\curveto(210,107.90860262)(227.90861001,89.99999262)(250,89.99999262)
\curveto(272.09138999,89.99999262)(290,107.90860262)(290,129.99999262)
\closepath
}
}
{
\pscustom[linewidth=11,linecolor=white]
{
\newpath
\moveto(346.49664,180.00001328)
\curveto(346.49664,88.04645811)(271.95355517,13.50337328)(180,13.50337328)
\curveto(88.04644483,13.50337328)(13.50336,88.04645811)(13.50336,180.00001328)
\curveto(13.50336,271.95356845)(88.04644483,346.49665328)(180,346.49665328)
\curveto(271.95355517,346.49665328)(346.49664,271.95356845)(346.49664,180.00001328)
\closepath
}
}
{
\pscustom[linewidth=1.4,linecolor=black,linestyle=dashed,dash=5.6 5.6]
{
\newpath
\moveto(340,180.00000047)
\curveto(340,268.36556044)(268.36555997,340.00000047)(180,340.00000047)
\curveto(91.63444003,340.00000047)(20,268.36556044)(20,180.00000047)
\curveto(20,91.6344405)(91.63444003,20.00000047)(180,20.00000047)
\curveto(268.36555997,20.00000047)(340,91.6344405)(340,180.00000047)
\closepath
}
}
{
\pscustom[linewidth=1,linecolor=gray,fillstyle=solid,fillcolor=gray]
{
\newpath
\moveto(260,299.99999477)
\curveto(260,294.47714727)(255.5228475,289.99999477)(250,289.99999477)
\curveto(244.4771525,289.99999477)(240,294.47714727)(240,299.99999477)
\curveto(240,305.52284226)(244.4771525,309.99999477)(250,309.99999477)
\curveto(255.5228475,309.99999477)(260,305.52284226)(260,299.99999477)
\closepath
\moveto(260,59.99999477)
\curveto(260,54.47714727)(255.5228475,49.99999477)(250,49.99999477)
\curveto(244.4771525,49.99999477)(240,54.47714727)(240,59.99999477)
\curveto(240,65.52284226)(244.4771525,69.99999477)(250,69.99999477)
\curveto(255.5228475,69.99999477)(260,65.52284226)(260,59.99999477)
\closepath
}
}
{
\pscustom[linewidth=2.8,linecolor=gray]
{
\newpath
\moveto(250,299.99999738)
\lineto(190,289.99999738)
\lineto(190,69.99999738)
\lineto(250,59.99999738)
\lineto(310,109.99999738)
\lineto(310,249.99999738)
\closepath
}
}
\end{pspicture}

%% file: picd.tex
\psset{xunit=.35pt,yunit=.35pt,runit=.35pt}
\begin{pspicture}(360,360)
{
\newgray{lightgrey}{0.9}
}
{
\pscustom[linestyle=none,fillstyle=solid,fillcolor=lightgray]
{
\newpath
\moveto(60,60.000003)
\lineto(180,180.000003)
\lineto(300,60.000003)
\lineto(230,20.000003)
\lineto(130,20.000003)
\closepath
}
}
{
\pscustom[linestyle=none,fillstyle=solid,fillcolor=lightgrey]
{
\newpath
\moveto(60,300)
\lineto(180,180.000003)
\lineto(300,300)
\lineto(205,350)
\lineto(155,350)
\closepath
}
}
{
\pscustom[linewidth=1.4,linecolor=black]
{
\newpath
\moveto(60,300.000003)
\lineto(300,60.000003)
\moveto(60,60.000003)
\lineto(140,140.000003)
\moveto(220,220.000003)
\lineto(300,300.000003)
}
}
{
\pscustom[linewidth=11,linecolor=white]
{
\newpath
\moveto(346.49664,180.00001628)
\curveto(346.49664,88.04646111)(271.95355517,13.50337628)(180,13.50337628)
\curveto(88.04644483,13.50337628)(13.50336,88.04646111)(13.50336,180.00001628)
\curveto(13.50336,271.95357145)(88.04644483,346.49665628)(180,346.49665628)
\curveto(271.95355517,346.49665628)(346.49664,271.95357145)(346.49664,180.00001628)
\closepath
}
}
{
\pscustom[linewidth=1.4,linecolor=black,linestyle=dashed,dash=5.6 5.6]
{
\newpath
\moveto(340,180.00000347)
\curveto(340,268.36556344)(268.36555997,340.00000347)(180,340.00000347)
\curveto(91.63444003,340.00000347)(20,268.36556344)(20,180.00000347)
\curveto(20,91.6344435)(91.63444003,20.00000347)(180,20.00000347)
\curveto(268.36555997,20.00000347)(340,91.6344435)(340,180.00000347)
\closepath
}
}
\end{pspicture}